\documentclass[10pt]{article}
\usepackage[a4paper, total={5.5 in, 10in}]{geometry}
\usepackage[utf8]{inputenc}
\usepackage{amssymb}
\usepackage{amsthm}
\usepackage{amsmath}
\usepackage{mathtools}
\usepackage[T1]{fontenc}
\usepackage{fancyhdr}
\usepackage{currvita}
\usepackage{xcolor}
\usepackage{enumitem}

\newtheorem{theorem}{Theorem}[section]
\newtheorem{prop}[theorem]{Proposition}
\newtheorem{dfn}[theorem]{Definition}

\newcommand\Z{\mathbb Z}
\newcommand\R{\mathbb R}

\newcommand\s{\sigma}
\newcommand\rt{\rtimes}

\newcommand{\Irr}{\mathop{\text{Irr}}}

\newcommand{\Lad}{\mathop{\text{Lad}}}

\numberwithin{equation}{section}
\title{\Large\textbf{Representations induced from cuspidal and ladder representations of classical $p$-adic groups}}
\date{}
\author{Barbara Bo\v{s}njak}
\begin{document}
\maketitle
\thispagestyle{fancy}
\fancyhf{}
\lfoot{
This work has been supported in
part by Croatian Science Foundation under the
project IP-01-2018-3628.}
\begin{abstract}
Let $G_n$ denote either the group $Sp(2n, F)$ or $SO(2n+1, F)$ over a non-archimedean local field $F$. We determine the composition series of representations of $G_n$ induced from cuspidal and ladder representations such that the minimal exponent in the cuspidal support of the ladder representation is greater than or equal to $\frac{1}{2}$.
\end{abstract}

\section{\large{Introduction}}
The essentially Speh representations have played an important role in the classification of the unitary dual of a general linear group over the $p$-adic field obtained in \cite{Tad9}. 
The ladder representations are a generalization of the essentially Speh representations.
They were first studied in \cite{LapidMinguez,LapidMinguez1}. The importance of ladder representations is reflected in the way they simplify the proof of the unitary dual classification. Additionally, the class of ladder representations has an advantage over the class of essentially Speh representations on the level of Jacquet modules. 
The results of \cite{KretLapid} show that the constituents of the Jacquet modules of ladder representations are tensor products of representations of the same type.
One can easily see that an analogue is generally not true for the Jacquet modules of the essentially Speh representations.
Since our approach in establishing the irreducible subquotients of an induced representation is based on the 
Geometric Lemma, we consider the class of ladder representations as a more accessible one. 

Motivated by the structure of representations in the unitary dual of a general linear group,
we study the composition series of representations induced from a ladder representation on the general linear part and a unitary representation on the classical part, which should provide insight into the structure of the unitary dual of other classical $p$-adic groups. 
The reducibility of such induced representations is known if the representation on the classical part is cuspidal, see \cite{LT}. The goal of this paper is to obtain the composition series when the cuspidal support of the ladder representation contains only representations of the form $\nu^x\rho$ for a character $\nu(g)=|\text{det}(g)|_F$ of a general linear group and $x\ge\frac{1}{2}$.

The methods used in this paper rely on the structural formula for the Jacquet modules of induced representations obtained in \cite{Tad5}, which is an algebraization of the Geometric Lemma. Namely, together with the Langlands classification, it gives candidates for the subquotients of the considered induced representation.
All irreducible tempered representations which appear in our investigation happen to be strongly positive. We note that the strongly positive representations play an important role in determining which of the obtained candidates are indeed the subquotients of the considered induced representation. 
M\oe glin and Tadić have shown in \cite{Moe2,MT1} that strongly positive representations are the fundamental objects in the construction of discrete series for classical $p$-adic groups. Their definition justifies their name since the exponents in the GL-cuspidal support of any induced representation into which they embed are positive. 
Also, they are the only irreducible tempered representations which are completely determined by the cuspidal support.
In order to obtain subquotients of the considered induced representation we also use some sufficient conditions on the subquotients of certain Jacquet modules, mainly through the transitivity of the Jacquet modules and parabolic induction.       

We briefly outline the content of the paper. In the following section we introduce the notation that will be used throughout the paper. 
In the third section we determine the tempered irreducible subquotients of the representation induced from cuspidal and ladder representations such that the minimal exponent in the cuspidal support of the ladder representation is greater than or equal to $\frac{1}{2}$. In the fourth section we obtain a reducibility criterion based on the results in \cite{LT} and a description of the non-tempered irreducible subquotients of the considered representation. 

For the reader's convenience we give a theorem that describes the length of a considered representation. It is an easy consequence of the main results of the paper, stated as Theorems \ref{tm1} and \ref{tm2}, which describe the irreducible subquotients of the considered representation. 
In the following theorem, we consider the representation $\pi_L$, which is defined as the unique irreducible subrepresentation of
$$\delta([\nu^{a_1}\rho,\nu^{b_1}\rho])\times\ldots\times\delta([\nu^{a_t}\rho,\nu^{b_t}\rho])$$
for an irreducible cuspidal representation $\rho$ of the general linear group, a positive integer $t$ and real numbers $a_i, b_i$,  $i=1,\ldots,t$, such that $\frac{1}{2}\le a_1<\ldots<a_t$, $b_1<\ldots<b_t$ and $a_i\le b_i$ for $i=1,\ldots,t$. 
Also, we let $\sigma_c$ denote an irreducible cuspidal representation of the special odd orthogonal or symplectic group over a  non-archimedean local field.
\begin{theorem}
    Let $m$ be a number of indices $i\in\{1,\ldots,t\}$ such that there is an $x\in \{a_i, a_i+1,\ldots, b_i\}$ for which $\nu^x\rho\rtimes\sigma_c$ reduces. 
    The induced representation $\pi_L \rtimes \sigma_c$ is of the length $m+1$.
\end{theorem}
The author is thankful to Ivan Mati\'{c} for suggesting the problem and for many helpful comments and to Marcela Hanzer for several useful comments. Also, the author would like to thank Petar Baki\'{c} and Vanja Wagner for a number of suggestions and grammatical corrections. 
\section{\large{Preliminaries}} \label{s2}
Throughout the paper, $F$ will denote a non-archimedean local field of characteristic different than two.

We will now describe the groups that we consider. We denote by $GL(n, F)$ the general linear group of rank $n$ over $F$. In what follows, we shall fix one of the series of classical groups $SO(2n+1, F)$ and $Sp(2n, F)$ and denote by $G_{n}$ a rank $n$ group belonging to this fixed series. The matrix realizations of these groups can be found in \cite[Sections 3,6]{Tad5}.

We fix a set of standard parabolic subgroups of $G_n$ by choosing a minimal $F$-parabolic subgroup in $G_{n}$ consisting of upper-triangular matrices.
Standard parabolic subgroups have the Levi decomposition $P=MN$, where $N$ is its unipotent radical and $M$ is the Levi factor of the form $GL(n_{1}, F) \times \ldots \times GL(n_{k}, F) \times G_{n'}$ where $\alpha=(n_1,\ldots,n_k)$ such that $n_1+\ldots+n_k\le n$. We use the notation $P=P_{\alpha}$, $M=M_{\alpha}$, $N=N_{\alpha}$.
For representations $\delta_{i}$ of $GL(n_{i}, F)$, $i = 1, 2, \ldots, k$ and a representation  $\tau$ of $G_{n'}$, we denote the normalized induction $\text{Ind}_{M}^{G_n}(\delta_{1} \otimes \ldots \otimes \delta_{k} \otimes \tau)$ by $\delta_{1} \times \ldots \times \delta_{k} \rtimes \tau$. 

The set of all irreducible admissible representations of $GL(n,F)$ (resp.~$G_n$) will be denoted by $\Irr(GL(n,F))$ (resp.~$\Irr(G_n)$). We denote $\Irr(G)=\cup_{n\ge0} \Irr(G_n)$ (resp.~$\Irr(GL)=\cup_{n\ge0} \Irr(GL(n,F))$).  
Let $R(G_n)$ denote the Grothendieck group of admissible, complex representations of finite length of $G_n$ and define $R(G)=\oplus_{n \geq 0} R(G_{n})$. Similarly, $R(GL) = \oplus_{n \geq 0} R(GL(n, F))$. For $\pi_1,\pi_2\in\Irr(G_n)$ the fact that $\pi_1$ is a subquotient of $\pi_2$ is denoted by $\pi_1\le\pi_2$. We also say that $\pi_1$ is contained in $\pi_2$.
The cuspidal support of an admissible representation $\pi$ of $GL(n,F)$ or $G_n$ is denoted by $[\pi]$.
The contragredient of a representation $\pi$ in $R(GL)$ or $R(G)$ is denoted by $\widetilde{\pi}$. We say $\pi$ is a selfcontragredient representation if $\widetilde{\pi}\simeq\pi$.
For $\pi,\pi_1,\pi_2\in \Irr(GL)$ and $\s\in\Irr(G)$ we have $\pi \rt \s = \widetilde{\pi} \rt \s$ and $\pi_1 \times \pi_2 \rt \s = \pi_2 \times \pi_1 \rt \s$ in $R(G)$.

For $\sigma \in \Irr(G_{n})$ and $\alpha=(n_1,\ldots,n_k)$ such that $n_1+\ldots+n_k\le n$ we denote by $r_{\alpha}(\sigma)$ the normalized Jacquet module of $\sigma$ with respect to the parabolic subgroup $P_{\alpha}$ with the Levi subgroup equal to $GL(n_1, F) \times \ldots \times GL(n_k,F)\times G_{n'}$. For $\alpha=(m)$ we identify $r_{(m)}(\sigma)$ with its semisimplification in $R(GL(m,F)) \otimes R(G_{n-m})$ and consider
\begin{equation*}
\mu^{\ast}(\sigma) = 1 \otimes \sigma + \sum_{m=1}^{n} r_{(m)}(\sigma) \in R(GL) \otimes R(G).
\end{equation*}
Similarly, for $\pi\in\Irr(GL)$, we denote by  $r_{(m)}(\pi)$ the normalized Jacquet module of $\pi$ with respect to the standard parabolic subgroup having the Levi factor equal to $GL(m, F) \times GL(n-m,F)$. We identify it with its semisimplification and define $m^*(\pi)=\sum_{m=0}^n r_{(m)}(\pi)$.
The mappings $\mu^*$ and $m^*$ extend to $R(G)$ and $R(GL)$ respectively. We define $M^*:R(GL)\to R(GL)\otimes R(GL)$ by 
$$M^*=(m\otimes 1)\circ (\sim\otimes m^*)\circ s \circ m^*$$
where $m$ denotes $\times:R(GL)\otimes R(GL)\to R(GL)$ and $s: \sum x_i\otimes y_i \mapsto \sum y_i\otimes x_i$.
The following theorem is the main result of \cite{Tad5}.
\begin{theorem}
For $\pi\in R(GL)$ and $\sigma\in R(G)$
$$\mu^*(\pi\rtimes\sigma)=M^*(\pi)\rtimes\mu^*(\sigma) $$
where the right hand side is determined by $(\pi_1\otimes\pi_2)\rtimes(\pi'\otimes\sigma')=(\pi_1\times\pi')\otimes(\pi_2\rtimes\sigma')$.
\end{theorem}
The Geometric Lemma implies that $m^*$ is multiplicative, i.e. $m^*(\pi_1\times\pi_2)=m^*(\pi_1)\times m^*(\pi_2)$. Together with the definition of $M^*$, this implies $M^*$ is multiplicative.

We recall the Frobenius reciprocity:
$$ \text{Hom}_{G_n}(\pi,\delta_1\times\ldots\times\delta_k\rtimes\tau)\cong \text{Hom}_{M_{\alpha}}(r_{\alpha}(\pi),\delta_1\otimes\ldots\otimes\delta_k\otimes\tau). $$
Here $P_{\alpha} = M_{\alpha}N_{\alpha}$ is the standard parabolic subgroup of $G_n$ corresponding to the $k$-tuple $\alpha$.

The character $|det(g)|_F$ of $GL(n,F)$ is denoted by $\nu$, where $|\text{ }|_F$ is the normalized absolute value on $F$. 
If $x,y$ are real numbers such that $y-x\in\mathbb{N}_0$, we denote the set $\{x,x+1,\ldots,y\}$ by $[x,y]$. 
Let $\rho$ be an irreducible cuspidal representation of $GL(n, F)$. 
Similarly, we refer to the set $\{ \rho, \nu \rho, \ldots, \nu^{m} \rho \}$ as a segment of cuspidal representations, and denote it by $[ \rho, \nu^{m} \rho]$.
Attached to each such segment is an irreducible essentially square-integrable representation $\delta ([\rho, \nu^{m} \rho ])$ of $GL(m \cdot n, F)$, which is the unique irreducible subrepresentation of $\nu^{m} \rho \times \ldots \times \nu \rho \times \rho$. 
It is well-known that every irreducible essentially square-integrable representation $\delta \in \Irr(GL(n,F))$ is of the form $\delta([\nu^{x} \rho, \nu^{y} \rho])$, for an irreducible cuspidal representation $\rho \in \Irr(GL(n',F))$ and $x, y \in \mathbb{R}$ such that $y - x$ is a non-negative integer.
For an irreducible essentially square-integrable representation $\delta$, there is a unique $e(\delta) \in \mathbb{R}$ such that $\nu^{- e(\delta)} \delta$ is unitarizable. Note that $e(\delta([\nu^{x} \rho, \nu^{y} \rho])) = \frac{x+y}{2}$. We denote by $1$ the one dimensional representation of the trivial group and put $\delta([\nu^{x} \rho, \nu^{y} \rho]) = 1$ if $y=x-1$.

In the paper we favour the subrepresentation over the quotient version of the Langlands classification, which we now recall. Let $\delta_1, \ldots, \delta_l $ be the irreducible essentially square-integrable representations of general linear groups.
If $e(\delta_1) \leq \ldots \leq e(\delta_l)$, then $\delta_1\times\ldots\times\delta_l$ has a unique irreducible (Langlands) subrepresentation which we denote by $L(\delta_1,\ldots,\delta_l)$. 
If $\tau \in \Irr(G_{n'})$ is a tempered representation and $e(\delta_1) \leq \ldots \leq e(\delta_l) < 0$, then $\delta_{1} \times \ldots \times \delta_{l} \rtimes \tau$ has a unique irreducible (Langlands) subrepresentation which we denote by $L(\delta_{1}, \ldots, \delta_{l} ; \tau)$.
Every irreducible non-tempered representation of $GL(n,F)$ (resp. $G_{n'}$) is equal to $L(\delta_1,\ldots,\delta_l)$ (resp. $L(\delta_{1}, \ldots, \delta_{l} ; \tau)$) for a unique choice of $\delta_1,\ldots,\delta_l$ (and $\tau$).

Throughout the paper we fix cuspidal representations $\rho\in\Irr(GL(n,F))$ and $\s_c\in\Irr(G_{n'})$. 
\begin{dfn}
A ladder representation is a representation of the form $$L(\delta([\nu^{x_1}\rho,\nu^{y_1}\rho]),\ldots,\delta([\nu^{x_k}\rho,\nu^{y_k}\rho]))$$
for a positive integer $k$, real numbers $x_1,\ldots,x_k$ and $y_1,\ldots,y_k$ such that $x_1<\ldots<x_k$, $y_1<\ldots<y_k$ and $x_i\le y_i$ for $i=1,\ldots,k$.
\end{dfn}
Now let $\pi=L(\delta([\nu^{x_1}\rho,\nu^{y_1}\rho]),\ldots,\delta([\nu^{x_k}\rho,\nu^{y_k}\rho]))$ be a ladder representation. We will frequently need to consider the Jacquet modules of $\pi \rtimes \sigma_c$ taken with respect to a standard maximal parabolic subgroup. From \cite[Theorem 2.1]{KretLapid} we see
\begin{gather}
    \label{lad}
    m^*(\pi)= 
    \displaystyle \sum_{\Lad(\pi)} L(\delta([\nu^{c_1+1}\rho,\nu^{y_1}\rho]),\ldots,\delta([\nu^{c_k+1}\rho,\nu^{y_k}\rho]))\otimes  \\ \otimes L(\delta([\nu^{x_1}\rho,\nu^{c_1}\rho]),\ldots,\delta([\nu^{x_k}\rho,\nu^{c_k}\rho])) \nonumber
\end{gather}
where $\Lad(\pi)$ denotes the set of all $k$-tuples $(c_1,\ldots,c_k)$ of real numbers for which the following conditions hold:
\begin{itemize}
    \item $c_1<\ldots<c_k$,
    \item $x_i-1 \le c_i \le y_i$ for $i=1,\ldots,k$,
    \item $c_i-x_i\in\Z$ for $i=1,\ldots,k$.
\end{itemize}
We define $\Lad(\pi)'$ as the set of all pairs $\big((c_1,\ldots,c_k),(d_1,\ldots,d_k)\big)$ in $\Lad(\pi)\times\Lad(\pi)$ such that $c_i\le d_i$ for $i=1,2,\ldots,k$.
Combining (\ref{lad}) with \cite[Theorem 5.4]{Tad5} we have
\begin{gather}
    \label{struct}
    \mu^*(\pi\rtimes\sigma_c)=  \sum\limits_{\Lad(\pi)'} L(\delta([\nu^{-c_k}\widetilde{\rho},\nu^{-x_k}\widetilde{\rho}]),\ldots,\delta([\nu^{-c_1}\widetilde{\rho},\nu^{-x_1}\widetilde{\rho}]))\times \\
    \times L(\delta([\nu^{d_1+1}\rho,\nu^{y_1}\rho]),\ldots,\delta([\nu^{d_k+1}\rho,\nu^{y_k}\rho])) \otimes  \nonumber \\
    L(\delta([\nu^{c_1+1}\rho,\nu^{d_1}\rho]),\ldots,\delta([\nu^{c_k+1}\rho,\nu^{d_k}\rho])) \rtimes \s_c \nonumber.
\end{gather}
For a positive integer $t$ and real numbers $a_1,\ldots,a_t$ and $b_1,\ldots,b_t$ such that $ \frac{1}{2}\le a_1<\ldots<a_t$, $b_1<\ldots<b_t$ and $a_i\le b_i$ for $i=1,\ldots,t$, we define the ladder representation $$\pi_L=L(\delta([\nu^{a_1}\rho,\nu^{b_1}\rho]),\ldots,\delta([\nu^{a_t}\rho,\nu^{b_t}\rho])).$$ 
The goal of this article is to determine the composition series of $\pi_L\rtimes\s_c$.

For selfcontragredient $\rho$, there is a unique non-negative real number $x$ such that $\nu^x\rho\rtimes\sigma_c$ reduces. Otherwise, $\nu^x\rho\rtimes\sigma_c$ is irreducible for every real number $x$.

An irreducible representation $\sigma$ of $G_n$ is called
strongly positive if
\begin{equation*}
\sigma \hookrightarrow \nu^{x_{1}} \rho_{1} \times \ldots \times \nu^{x_{k}} \rho_{k} \rtimes \sigma_{cusp},
\end{equation*}
(where $\rho_i$, $\sigma_{cusp}$ are unitary cuspidal representations) implies $x_1,\ldots ,x_k$ are positive. By \cite{Arthur} and \cite[Th\'{e}or\`{e}me~3.1.1]{Moe2}, if $\nu^{x} \rho$ appears in the cuspidal support of $\sigma$, then $\rho$ is selfcontragredient and $2x \in \mathbb{Z}$.

Strongly positive representations play an important role in the classification of discrete series by M{\oe}glin-Tadi\'{c} \cite{Moe2,MT1}. 
The classification of irreducible strongly positive square-integrable genuine representations of metaplectic groups over $F$ is obtained in \cite[Theorem 1.2]{Matic3} using a purely algebraic approach. 
These results can be applied in the case of classical $p$-adic groups in a completely analogous manner. Every strongly positive representation whose cuspidal support is contained in $\{\nu^x\rho : x\in\R\}$ for some irreducible cuspidal selfcontragredient representation $\rho\in \Irr(GL(n,F))$ can be realized in the unique way as the unique irreducible subrepresentation of the induced representation
\begin{equation*} 
 \delta([\nu^{\alpha -r +1 } \rho, \nu^{x_{1}} \rho])\times\ldots\times\delta([\nu^{\alpha } \rho, \nu^{x_{r}} \rho])  \rtimes \sigma_{cusp}
\end{equation*}
where $\sigma_{cusp}$ is an irreducible cuspidal representation of $G_{n'}$, $\alpha$ is the unique positive number such that $\nu^{\alpha} \rho \rtimes \sigma_{cusp}$ reduces,
$r = \lceil \alpha \rceil$ and $-\frac{1}{2} \le x_1 < x_2 < \ldots < x_r$, $x_i - \alpha \in \mathbb{Z}$ for $i = 1, \ldots,r$. Here $\lceil \alpha \rceil$ denotes the smallest integer which is not smaller than $\alpha$. 

\section{\large{Tempered subquotients of $\pi_L\rtimes\s_c$}}
In this section we determine the tempered subquotients of $\pi_L\rtimes\s_c$. These results will be applied in the following section, with the goal of determining non-tempered subquotients of $\pi_L\rtimes\s_c$.
Namely, by the Langlands classification,  every non-tempered subquotient of $\pi_L\rtimes\sigma_c$ is of the form $L(\delta_1,\ldots,\delta_k;\tau)$.
If $L(\delta_1,\ldots,\delta_k;\tau)\le\pi_L\rtimes\s_c$, then the transitivity of Jacquet modules implies $L(\delta_1,\ldots,\delta_k)\otimes\tau\le\mu^*(\pi_L\rtimes\s_c)$. Now (\ref{struct}) implies that there exists $\big((c_1,\ldots,c_t),(d_1,\ldots,d_t)\big)\in \Lad(\pi_L)'$ such that $$\tau\le   L(\delta([\nu^{c_1+1}\rho,\nu^{d_1}\rho]),\ldots,\delta([\nu^{c_t+1}\rho,\nu^{d_t}\rho])) \rtimes \s_c.$$
Note that the exponents in the cuspidal support of a ladder representation on the general linear part are positive since $\frac{1}{2}\le a_1\le c_1+1$.

Note that any tempered subquotient of $\pi_L\rtimes\s_c$ is necessarily strongly positive.
To see this, suppose that $\pi_L \rtimes \sigma_c$ contains an irreducible tempered subquotient $\tau$ which is not strongly positive. Then
there exists an irreducible tempered representation $\tau'$ and $a,b\in\R$, $a\leq 0$, $a + b \geq 0$ such that $\tau$ is a subrepresentation of $\delta([\nu^a \rho,
\nu^b \rho]) \rtimes \tau'$.
If $a$ and $b$ are integers, then $\rho \in [\pi_L \rtimes \sigma_c]$, which is a contradiction. 
If $a$ and $b$ are half-integers, then $[\pi_L \rtimes \sigma_c]$ contains the multiset $\{\nu^{\frac{1}{2}}\rho,\nu^{\frac{1}{2}}\rho\}$, which is again a contradiction.

The following theorem determines when $\pi_L\rtimes\sigma_c$ contains a tempered subquotient. 
\begin{theorem}
\label{tm1}
The induced representation $\pi_L\rtimes\s_c$ has a tempered irreducible subquotient if and only if there exists $\alpha>0$ such that $\nu^{\alpha}\rho\rtimes\sigma_c$ reduces and $a_i=\alpha-t+i$ for $i=1,2,\ldots,t$.

In that case the tempered irreducible subquotient is unique; it is the unique irreducible subrepresentation of $\pi_L\rtimes\s_c$ and is isomorphic to the  strongly positive representation characterised as the unique irreducible subrepresentation of 
\begin{gather}
    \delta([\nu^{\alpha -t +1 } \rho, \nu^{b_{1}} \rho])\times\ldots\times\delta([\nu^{\alpha } \rho, \nu^{b_{t}} \rho])  \rtimes \sigma_{c}. \label{zz}
\end{gather}
\end{theorem}
\begin{proof}
Assume that the induced representation $\pi_L\rtimes\s_c$ has a tempered irreducible subquotient. We denote it by $\tau$.
We have noted above that $\tau$ is necessarily a strongly positive representation. Then there exists a unique positive real number $\alpha$ such that $\nu^{\alpha}\rho\rtimes\sigma_c$ reduces, as seen in section \ref{s2}. Note that this implies $\rho\simeq \widetilde{\rho}$. We denote $r=\lceil\alpha\rceil$.
Thus there exists the unique increasing sequence of real numbers $x_1,\ldots,x_r$ such that $\alpha-r+i-1\le x_i$ and $x_i-\alpha\in\Z$ for $i=1,2,\ldots,r$ and $\tau$ is the unique irreducible subrepresentation of 
\begin{equation}
\label{tau}
\delta([\nu^{\alpha -r +1 } \rho, \nu^{x_{1}} \rho])\times\ldots\times\delta([\nu^{\alpha } \rho, \nu^{x_{r}} \rho])  \rtimes \sigma_{c}.
\end{equation}
Therefore, it is also a subrepresentation of 
$L(\delta([\nu^{\alpha -r +1 } \rho, \nu^{x_{1}} \rho]),\ldots,\delta([\nu^{\alpha } \rho, \nu^{x_{r}} \rho]) ) \rtimes \sigma_{c}.$ 
The Frobenius reciprocity and the transitivity of Jacquet modules imply that $$L(\delta([\nu^{\alpha -r +1 } \rho, \nu^{x_{1}} \rho]),\ldots,\delta([\nu^{\alpha } \rho, \nu^{x_{r}} \rho]) ) \otimes \sigma_{c}\le \mu^*(\pi_L\rtimes\s_c).$$
By the structural formula (\ref{struct}) there exists $\big((c_1,\ldots,c_t),(d_1,\ldots,d_t)\big)\in \Lad(\pi_L)'$ such that 
\begin{gather}
\label{prva}
    L(\delta([\nu^{\alpha -r +1 } \rho, \nu^{x_{1}} \rho]),\ldots,\delta([\nu^{\alpha } \rho, \nu^{x_{r}} \rho]) ) \le  \\
    L(\delta([\nu^{-c_t}\rho,\nu^{-a_t}\rho]),\ldots,\delta([\nu^{-c_1}\rho,\nu^{-a_1}\rho]))
    \times \nonumber \\ \times L(\delta([\nu^{d_1+1}\rho,\nu^{b_1}\rho]),\ldots,\delta([\nu^{d_t+1}\rho,\nu^{b_t}\rho])) \nonumber
\end{gather}
and 
\begin{gather}
\label{druga}
    \s_c\le  L(\delta([\nu^{c_1+1}\rho,\nu^{d_1}\rho]),\ldots,\delta([\nu^{c_t+1}\rho,\nu^{d_t}\rho])) \rtimes \s_c.
\end{gather}
Note that the exponents in a cuspidal support of a representation on the left hand side of (\ref{prva}) are positive and $-a_i,-c_i<0$ for indices $i\in\{1,2,\ldots,t\}$ for which $\delta([\nu^{-c_i}\rho,\nu^{-a_i}\rho])\not\simeq 1$. Since representations in (\ref{prva}) have the same cuspidal support, we conclude that $c_i=a_i-1$ for $i=1,2,\ldots,t$. 
From (\ref{druga}) we get $c_i=d_i$ for $i=1,2,\ldots,t$, since $\sigma_c$ is a cuspidal representation. Thus (\ref{prva}) gives
$$ L(\delta([\nu^{\alpha -r +1 } \rho, \nu^{x_{1}} \rho]),\ldots,\delta([\nu^{\alpha } \rho, \nu^{x_{r}} \rho]) ) \le 
 L(\delta([\nu^{a_1}\rho,\nu^{b_1}\rho]),\ldots,\delta([\nu^{a_t}\rho,\nu^{b_t}\rho])).$$
Since these representations are irreducible, we have the equality. Langlands classification now implies equality of the sets consisting of the essentially square-integrable representations defining them. Thus we have $a_i=\alpha-t+i$ and $b_i=x_{r-t+i}$ for $i=1,2,\ldots,t$.

If $a_i=\alpha-t+i$ for $i=1,2,\ldots,t$, then $$\pi_L\cong L(\delta([\nu^{\alpha -t +1 } \rho, \nu^{b_{1}} \rho]),\ldots,\delta([\nu^{\alpha } \rho, \nu^{b_{t}} \rho]) ). $$
Note that $\pi_L\rtimes\s_c$ is a subrepresentation of the representation (\ref{zz}). Since there exists $\alpha>0$ such that $\nu^{\alpha}\rho\rtimes\sigma_c$ reduces, we conclude that the representation (\ref{zz}) contains a unique irreducible tempered subquotient, denoted $\tau$. 
It is a strongly positive representation and the unique irreducible subrepresentation of (\ref{zz}), hence it is also a subrepresentation of $\pi_L\rtimes\s_c$. 
Every subquotient of $\pi_L\rtimes\sigma_c$ is a subquotient of representation (\ref{zz}), so $\tau$ appears in composition series of $\pi_L\rtimes\sigma_c$ with multiplicity one.
\end{proof}
\section{\large{Non-tempered subquotients of $\pi_L\rtimes\s_c$}}
In the last section we will determine non-tempered subquotients of $\pi_L\rtimes\s_c$ and thus obtain the complete composition series of a representation $\pi_L\rtimes\s_c$. 
Firstly, we give a characterization of reducibility of the induced representation $\pi_L\rtimes\sigma_c$ based on the results in \cite{LT}.
\begin{prop}
\label{propi}
The induced representation $\pi_L\rtimes\sigma_c$ is reducible if and only if $\nu^{x}\rho\rtimes\sigma_c$ reduces for some $\nu^x\rho\in [\pi_L]$.
\end{prop}
\begin{proof}
  By Theorem 1.1 in \cite{LT}, if $\nu^{x}\rho\rtimes\sigma_c$ reduces for some $\nu^x\rho\in [\pi_L]$, then $\pi_L\rtimes\sigma_c$ is reducible.
  
  Conversely, assume that $\nu^{x}\rho\rtimes\sigma_c$ is irreducible for all $\nu^x\rho\in [\pi_L]$. 
  Since $\pi_L$ is a ladder representation, Theorem 1.2 in \cite{LT} implies that the induced representation $\pi_L\rtimes\sigma_c$ is irreducible if and only if $(\pi_{L})_{+}\times(\widetilde{\pi_L})_+$ is irreducible. 
  Here we use the notation from \cite[Definition 3.12.]{LT}: for a representation $\pi\simeq L(\delta(\Delta_1),\ldots,\delta(\Delta_k))$, let $(\pi)_+$ denote
  $L(\delta(\Delta_{i_1}),\ldots,\delta(\Delta_{i_l}))$, where $\{i_1,\ldots,i_l\}$ is a subset of $\{1,2,\ldots,k\}$ of all indices such that the sum of minimal and maximal exponent in $\Delta_i$ is positive. 
  From our assumption $a_1\ge\frac{1}{2}$ it follows that $(\pi_L)_+\simeq\pi_L$ and from \cite[Lemma 4.2]{Matic3} we see $\widetilde{\pi_L}\simeq L(\delta([\nu^{-b_t}\widetilde{\rho},\nu^{-a_t}\widetilde{\rho}]),\ldots,\delta([\nu^{-b_1}\widetilde{\rho},\nu^{-a_1}\widetilde{\rho}]))$ so $(\widetilde{\pi_L})_+\simeq1$. 
  Since $\pi_L$ is an irreducible representation, we get that the induced representation $\pi_L\rtimes\sigma_c$ is irreducible, thus finishing the proof.
\end{proof}
Since $\nu^{x}\rho\rtimes\sigma_c$ is irreducible for every $\nu^x\rho\in [\pi_L]$ for non-selfcontragredient $\rho$, the induced representation $\pi_L\rtimes\s_c$ is in that case irreducible and isomorphic to $$L(\delta([\nu^{-b_t}\rho,\nu^{-a_t}\rho]),\ldots,\delta([\nu^{-b_1}\rho,\nu^{-a_1}\rho]);\sigma_c).$$
For selfcontragredient $\rho$, there is a unique non-negative real number $x$ such that $\nu^x\rho\rtimes\sigma_c$ reduces and we denote it by $\alpha$. 
From Proposition \ref{propi} we see that for $\alpha=0$ the induced representation $\pi_L\rtimes\s_c$ is irreducible. 
Thus, in what follows we assume $\rho$ is selfcontragredient and $\alpha>0$. 

Recall that the non-tempered subquotients of induced representations from $R(G)$ are, according to the Langlands classification, parametrized by $L(\delta_1,\ldots,\delta_l;\tau)$, where $\delta_1, \ldots, \delta_l $ are irreducible essentially square-integrable representations of general linear groups such that $e(\delta_1) \leq \ldots \leq e(\delta_l) < 0$ and  $\tau \in \Irr(G_{n'})$ is a tempered representation.
Since $L(\delta_1,\ldots,\delta_l;\tau)\le\pi_L\rtimes\s_c$ and  $L(\delta_1,\ldots,\delta_l)\otimes\tau\le \mu^*(L(\delta_1,\ldots,\delta_l;\tau))$, the transitivity of the Jacquet modules and the structural formula (\ref{struct}) imply that there exists a pair 
$\big((c_1,\ldots,c_t),(d_1,\ldots,d_t)\big)\in \Lad(\pi_L)'$ such that 
\begin{gather}
    \label{a}
    L(\delta_1,\ldots,\delta_l)\le 
     L(\delta([\nu^{-c_t}\rho,\nu^{-a_t}\rho]),\ldots,\delta([\nu^{-c_1}\rho,\nu^{-a_1}\rho]))\times \\
    \times L(\delta([\nu^{d_1+1}\rho,\nu^{b_1}\rho]),\ldots,\delta([\nu^{d_t+1}\rho,\nu^{b_t}\rho])) \nonumber
\end{gather}
and
\begin{gather}
\label{b}
    \tau\le L(\delta([\nu^{c_1+1}\rho,\nu^{d_1}\rho]),\ldots,\delta([\nu^{c_t+1}\rho,\nu^{d_t}\rho])) \rtimes \s_c.
\end{gather}
Since $e(\delta_i)<0$ for $i=1,2,\ldots,l$, from (\ref{a}) we conclude $d_i=b_i$ for $i=1,2,\ldots,t$.
Namely, considering the cuspidal support of the representation on the right-hand side of (\ref{a}), we notice that there is no segment with both positive and negative exponents contained in it. 
Recall that assumed restrictions on the exponents give $-a_1\le -\frac{1}{2}$ and $d_1+1\ge c_1+1\ge a_1 \ge\frac{1}{2}$. If $d_1+1\ge\frac{3}{2}$, then the difference between the maximal negative and minimal positive exponent is greater than or equal to two. If $d_1+1=\frac{1}{2}$, then $c_1+1=a_1=\frac{1}{2}$, so $\delta([\nu^{-c_1}\rho,\nu^{-a_1}\rho])\simeq 1$. Now the maximal negative exponent is less than or equal to $-a_2$, which is less than or equal to  $-\frac{3}{2}$, so we have the same conclusion as in the first case. 

These conclusions reduce (\ref{a}) and (\ref{b}) to
\begin{gather*}
    L(\delta_1,\ldots,\delta_l)\le 
     L(\delta([\nu^{-c_t}\rho,\nu^{-a_t}\rho]),\ldots,\delta([\nu^{-c_1}\rho,\nu^{-a_1}\rho]))
\end{gather*}
and
\begin{gather} \label{t1}
    \tau\le L(\delta([\nu^{c_1+1}\rho,\nu^{b_1}\rho]),\ldots,\delta([\nu^{c_t+1}\rho,\nu^{b_t}\rho])) \rtimes \s_c,
\end{gather}
respectively.
Denote by $i_M$ (resp.~$i_m$) the maximal (resp.~minimal) index $i\in\{1,2,\ldots,t\}$ such that $\alpha\in [a_i, b_i]$.
Note that if $i_M$ exists, then $\alpha\in [a_i, b_i]$ for every $i\in[i_m,i_M]$, since $\alpha\in[a_{i_M},b_{i_m}]\subseteq[a_i,b_i]$. 
Further, denote by $k$ an arbitrary element of $[i_m,i_M]$.  We define $k_m$ to be a minimal index from $\{1,2,\ldots,t\}$ such that $\alpha-k+k_m\le b_{k_m}.$ 
Note that $\alpha-k+i\in[a_i,b_i],\forall i\in[k_m,k]$. 
This follows by inductively applying the following inequalities which are the consequence of $a_1<\ldots<a_t$ and $b_1<\ldots<b_t$:
\begin{gather*}
 b_{k_m+1}\ge b_{k_m}+1\ge \alpha-k+(k_m+1) \text{ and } a_{k-1}\le a_k-1\le \alpha-1.
\end{gather*}
Consequently, $\alpha-k+i$ are positive numbers for $i\in [k_m,k]$.
We define $\sigma_k$
as a unique irreducible subrepresentation of 
\begin{gather*}
    \delta([\nu^{\alpha-k+k_m}\rho,\nu^{b_{k_m}}\rho])\times\ldots\times\delta([\nu^{\alpha}\rho,\nu^{b_k}\rho]) \rtimes \s_c
\end{gather*}
and $\pi_k$, unless $k=t$ and $a_i=\alpha-t+i$ for $i=1,2,\ldots,t$, in which case we define it as
\begin{gather*}
    L(\delta([\nu^{-b_t}\rho,\nu^{-a_t}\rho]),\ldots,\delta([\nu^{-b_{k+1}}\rho,\nu^{-a_{k+1}}\rho]),\delta([\nu^{-\alpha+1}\rho,\nu^{-a_k}\rho]),\ldots, \\
    \delta([\nu^{-\alpha+k-k_m+1}\rho,\nu^{-a_{k_m}}\rho]),\delta([\nu^{-b_{k_m-1}}\rho,\nu^{-a_{k_m-1}}\rho]),\ldots,\delta([\nu^{-b_1}\rho,\nu^{-a_1}\rho]);\sigma_k).
\end{gather*}
Note that $\pi_k$ is well defined since $-b_{k+1}<-b_k<-\alpha+1$,  $b_{k_m-1}<\alpha-k+k_m-1$ because of the definition of index $k_m$ and $-(a_1+b_1)<0$.
\begin{theorem}
\label{tm2}
Assume that $\nu^{x}\rho\rtimes\sigma_c$ reduces for some $\nu^x\rho\in [\pi_L]$.
Then the induced representation $\pi_L\rtimes\s_c$ reduces. Furthermore, let $\alpha$ be the unique positive number such that $\nu^{\alpha}\rho\rtimes\sigma_c$ reduces. Then the semi-simplification of $\pi_L \rtimes \sigma_c$ is equal to 
\begin{itemize}
    \item[(i)] $\s_t + L(\delta([\nu^{-b_t}\rho,\nu^{-a_t}\rho]),\ldots,\delta([\nu^{-b_1}\rho,\nu^{-a_1}\rho]);\s_c) + \displaystyle\sum\limits_{k\in[i_m,t-1]} \pi_k,$ \\ if $a_i=\alpha-t+i$ for $i=1,2,\ldots,t$,
    \item[(ii)] $L(\delta([\nu^{-b_t}\rho,\nu^{-a_t}\rho]),\ldots,\delta([\nu^{-b_1}\rho,\nu^{-a_1}\rho]);\s_c) + \displaystyle\sum\limits_{k\in[i_m,i_M]} \pi_k, $ otherwise.
\end{itemize}
\end{theorem}
\begin{proof}
As seen in Proposition \ref{propi}, if $\nu^{x}\rho\rtimes\sigma_c$ reduces for some $\nu^x\rho\in [\pi_L]$, then the induced representation $\pi_L\rtimes\s_c$ reduces.
The existence of a tempered subquotient $\s_t$ is established in Theorem \ref{tm1}, so we turn to the non-tempered case.
Since $\nu^{\alpha}\rho\in [\pi_L]$, there exists $i\in\{1,\ldots,t\}$ such that $\alpha\in [a_i,b_i]$ so $[i_m,i_M]\neq\emptyset$, while it is possible for $[i_m,t-1]$ to be empty.

Firstly, assume that $\pi=L(\delta_1,\ldots,\delta_l;\tau)$ is a non-tempered subquotient of $\pi_L\rtimes\sigma_c$. 
From (\ref{t1}) we can see that $\tau$, if not cuspidal, is necessarily equal to $\sigma_k$ for some $k\in[i_m,i_M]$. 
Namely, Theorem \ref{tm1} gives the following conditions on real numbers $c_i$ defined in the discussion before Theorem \ref{tm2} since $c_1+1\ge a_1 \ge\frac{1}{2}$:
\begin{enumerate}
    \item For some fixed $k\in[i_m,i_M]$ such that $c_k+1=\alpha$, we get $c_i=b_i$ for $i=k+1,\ldots,t$. Similarly, since $\alpha-k+i>b_i$ for $i=1,2,\ldots,k_m-1$, we get $c_i=b_i$ for $i=1,2,\ldots,k_m-1$.
    \item There exists $k_0\in[k_m,k]$ such that $c_i=b_i$ for $i=k_m,\ldots,k_0-1$ and $c_i+1=\alpha-k+i$ for $i=k_0,\ldots,k$. We claim that $k_0=k_m$. Recall that $\alpha-k+i\in[a_i,b_i],\forall i\in[k_m,k]$.
    Thus if $k_0> k_m$, then $c_{k_0-1}=b_{k_0-1}$ and $c_{k_0}+1=\alpha-k+k_0$ imply 
    $$\alpha-k+k_0-2=c_{k_0}-1 \ge c_{k_0-1}=b_{k_0-1}\ge \alpha-k+k_0-1, $$
    which is a contradiction. 
\end{enumerate}
In this way, $\pi$ is either equal to $L(\delta([\nu^{-b_t}\rho,\nu^{-a_t}\rho]),\ldots,\delta([\nu^{-b_1}\rho,\nu^{-a_1}\rho]);\sigma_c)$ or to $\pi_k$ for some $k\in[i_m,i_M]$.

Conversely, we prove that $\pi_k$ is a subquotient of $\pi_L\rtimes\sigma_c$ for every $k\in [i_m,i_M]$. Let us fix some $k\in [i_m,i_M]$. 
We define $\delta_i=\delta([\nu^{-b_i}\rho,\nu^{-a_i}\rho])$ for $i=1,\ldots,k_m-1,k+1,\ldots,t$, $\delta_i=\delta([\nu^{-\alpha+k-i+1}\rho,\nu^{-a_{i}}\rho])$ for $i=k_m,\ldots,k$,
\begin{gather*}
    \pi_k^{(1)} = \delta_t\times\ldots\times\delta_{k+1}\times\delta_k\times \ldots
    \times \delta_{k_m}\times\delta_{k_m-1}\times\ldots\times\delta_1, \\
    \Pi = \pi_k^{(1)}\times L(\delta([\nu^{\alpha-k+k_m}\rho,\nu^{b_{k_m}}\rho]),\ldots,\delta([\nu^{\alpha}\rho,\nu^{b_k}\rho])) \rtimes \s_c
\end{gather*}
and $\pi_k^{(2)} = L(\delta_t,\ldots,\delta_1)$, i.e. the Langlands subrepresentation of $\pi_k^{(1)}$.
Note that $\sigma_k\hookrightarrow L(\delta([\nu^{\alpha-k+k_m}\rho,\nu^{b_{k_m}}\rho]),\ldots,\delta([\nu^{\alpha}\rho,\nu^{b_k}\rho])) \rtimes \s_c$
since $\sigma_k$ is the unique irreducible subrepresentation of $\delta([\nu^{\alpha-k+k_m}\rho,\nu^{b_{k_m}}\rho])\times\ldots\times\delta([\nu^{\alpha}\rho,\nu^{b_k}\rho]) \rtimes \s_c$. 
Thus together with $\pi_k^{(2)}\hookrightarrow\pi_k^{(1)}$, we obtain $\pi_k\le \Pi$. 
It is an easy consequence of Langlands classification that $$L(\delta([\nu^{a_{k_m}}\rho,\nu^{b_{k_m}}\rho]),\ldots,\delta([\nu^{a_k}\rho,\nu^{b_k}\rho]))$$ is a subquotient of 
$$ \widetilde{\delta}_{k_m}\times\ldots\times\widetilde{\delta}_k\times L(\delta([\nu^{\alpha-k+k_m}\rho,\nu^{b_{k_m}}\rho]),\ldots,\delta([\nu^{\alpha}\rho,\nu^{b_k}\rho])). $$
This, together with the properties of $R(G)$,  clearly implies $\pi_L\rtimes\sigma_c \le \Pi$.

The Frobenius reciprocity implies that $\pi_k^{(2)}\otimes\sigma_k$ is a subquotient of $\mu^*(\pi_k)$ and $\mu^*(\Pi)$.
Also note that the above discussion shows that we can choose real numbers $c_i$ for $i=1,\ldots,t$ to obtain $\pi_k^{(2)}\otimes\sigma_k$ in $\mu^*(\pi_L\rtimes\sigma_c)$.
The transitivity of Jacquet modules now implies that the representation $\delta_t\otimes\ldots\otimes\delta_1\otimes\sigma_k$ is a subquotient of $r_{\beta}(\pi_k)$, $r_{\beta}(\pi_L\rtimes\sigma_c)$ and $r_{\beta}(\Pi)$ for the appropriate parabolic subgroup $P_{\beta}$.

To finish the proof, it remains to see that $\delta_t\otimes\ldots\otimes\delta_1\otimes\sigma_k$ appears in $r_{\beta}(\Pi)$ with multiplicity one. 
Firstly, we want to determine all the constituents of $\mu^*(\Pi)$ of the form $\delta_t\otimes\pi$ for some $\pi\in \Irr(G)$. 
According to the definition of $\pi_L$, $-b_t$ is the smallest exponent in the cuspidal support of $\Pi$. Also, $\nu^{-b_t}\rho$ appears in $[\Pi]$ with multiplicity one. 
Using the multiplicativity of $M^{\ast}$, let $\pi_1 \otimes \pi_2$ denote an irreducible constituent of some factor of $\mu^{\ast}(\Pi)$ such that $\nu^{-b_t}\rho\in [\pi_1]$. We see that $\pi_1\otimes\pi_2$ is necessarily equal to $\delta_t\otimes 1\le  M^*(\delta_t)$.
Thus we get that
$$\delta_t\otimes \delta_{t-1}\times\ldots\times\delta_1\times L(\delta([\nu^{\alpha-k+k_m}\rho,\nu^{b_{k_m}}\rho]),\ldots,\delta([\nu^{\alpha}\rho,\nu^{b_k}\rho])) \rtimes \s_c $$
contains all constituents of $\mu^*(\Pi)$ of the form $\delta_t\otimes\pi$ for some $\pi\in \Irr(G)$. 
Inductively, $\delta_t\otimes\ldots\otimes\delta_1\otimes\sigma_k$ appears in $r_{\beta}(\Pi)$ with the same multiplicity as $\delta_k\otimes\ldots\otimes\delta_1\otimes\sigma_k$ in the Jacquet module of
\begin{gather*} 
    \Pi'=\delta_{k}\times\ldots\times\delta_1\times L(\delta([\nu^{\alpha-k+k_m}\rho,\nu^{b_{k_m}}\rho]),\ldots,\delta([\nu^{\alpha}\rho,\nu^{b_k}\rho])) \rtimes \s_c 
\end{gather*}
with respect to the appropriate parabolic subgroup.

Let us now determine all the constituents of $\mu^*(\Pi')$ of the form $\delta_k\otimes\pi'$ for some $\pi'\in \Irr(G)$. 
Note that $\nu^{-\alpha+1}\rho$ is not an element of the cuspidal support of $\delta_{k-1}\times\ldots\times\delta_{1}$ since minimal exponents in $[\delta_i]$ for $i=1,\ldots,k$ are in strictly increasing order. 
According to the formula (\ref{struct}) we search for $\nu^{-\alpha+1}\rho$ in   
\begin{gather*}
   L(\delta([\nu^{-x_{k}}\rho,\nu^{-\alpha}\rho]),\ldots,\delta([\nu^{-x_{k_m} }\rho,\nu^{-\alpha+k-k_m}\rho]))
\end{gather*}
where $x_{k_m}<\ldots<x_k$ and $\alpha-k+i-1 \le x_i\le b_i $ for $i=k_m,\ldots,k$. 
Thus $-x_{i_0}=-\alpha+1$ for some $i_0\in\{k_m,\ldots,k-1\}$ and $x_i=\alpha-k+i-1$ for $i=i_0+1,\ldots,k$. This implies $-\alpha+1=-x_{i_0}>-x_k=-\alpha+1$, showing it is not possible.
Using the multiplicativity of $M^{\ast}$,
let $\pi_1'\otimes\pi_2'$ denote an irreducible constituent of some factor of $\mu^{\ast}(\Pi')$ such that $\nu^{-\alpha+1}\rho\in [\pi_1']$. 
We see that $\pi_1'\otimes\pi_2'$ is necessarily equal to $\delta_k\otimes1\le M^*(\delta_k)$.
Inductively, $\delta_t\otimes\ldots\otimes\delta_1\otimes\sigma_k$ appears in $r_{\beta}(\Pi)$ with the same multiplicity as $\delta_{k_m-1}\otimes\ldots\otimes\delta_1\otimes\sigma_k$ in the Jacquet module of
\begin{gather*} 
    \Pi''=\delta_{k_m-1}\times\ldots\times\delta_1\times L(\delta([\nu^{\alpha-k+k_m}\rho,\nu^{b_{k_m}}\rho]),\ldots,\delta([\nu^{\alpha}\rho,\nu^{b_k}\rho])) \rtimes \s_c 
\end{gather*}
with respect to the appropriate parabolic subgroup.
Note that the greatest positive exponent in the cuspidal support of $\widetilde{\delta}_{k_m-1}\times\ldots\times\widetilde{\delta}_1$ is $b_{k_m-1}$, which is by definition of $k_m$ strictly less than $\alpha-k+k_m-1$.
Thus the cuspidal supports of representations  $\widetilde{\delta}_{k_m-1}\times\ldots\times\widetilde{\delta}_1$ and $L(\delta([\nu^{\alpha-k+k_m}\rho,\nu^{b_{k_m}}\rho]),\ldots,\delta([\nu^{\alpha}\rho,\nu^{b_k}\rho]))$ are disjoint multisets. 
The constituents of $\mu^*(\Pi'')$ of the form $\pi''\otimes\sigma_k$ for $\pi''\in \Irr(GL)$ with negative exponents in $[\pi'']$ are consequently contained in 
$$ \delta_{k_m-1}\times\ldots\times\delta_1\otimes L(\delta([\nu^{\alpha-k+k_m}\rho,\nu^{b_{k_m}}\rho]),\ldots,\delta([\nu^{\alpha}\rho,\nu^{b_k}\rho])) \rtimes \s_c .$$
Since $\sigma_k$ is a unique subrepresentation of 
$$L(\delta([\nu^{\alpha-k+k_m}\rho,\nu^{b_{k_m}}\rho]),\ldots,\delta([\nu^{\alpha}\rho,\nu^{b_k}\rho])) \rtimes \s_c,$$
we finally conclude that $\delta_t\otimes\ldots\otimes\delta_1\otimes\sigma_k$ appears in $r_{\beta}(\Pi)$ with multiplicity one. 
Hence it appears in $r_{\beta}(\pi_k)$ and $r_{\beta}(\pi_L\rtimes\s_c)$ with multiplicity one and we arrive to the conclusion that $\pi_k$ appears in $\pi_L\rtimes\sigma_c$ with multiplicity one. 
\end{proof}
\bibliographystyle{siam}
\bibliography{bibliografija}

\begin{thebibliography}{10}

\bibitem{Arthur}
{\sc J.~Arthur}, {\em The endoscopic classification of representations.
  Orthogonal and symplectic groups}, vol.~61 of American Mathematical Society
  Colloquium Publications, American Mathematical Society, Providence, RI, 2013.

\bibitem{KretLapid}
{\sc A.~Kret and E.~Lapid}, {\em Jacquet modules of ladder representations}, C.
  R. Math. Acad. Sci. Paris, 350 (2012), pp.~937--940.

\bibitem{LapidMinguez}
{\sc E.~Lapid and A.~M\'{i}nguez}, {\em On a determinantal formula of
  {T}adi\'c}, Amer. J. Math., 136 (2014), pp.~111--142.

\bibitem{LapidMinguez1}
\leavevmode\vrule height 2pt depth -1.6pt width 23pt, {\em On parabolic
  induction on inner forms of the general linear group over a non-archimedean
  local field}, Selecta Math. (N.S.), 22 (2016), pp.~2347--2400.

\bibitem{LT}
{\sc E.~Lapid and M.~Tadi{\'c}}, {\em Some results on reducibility of parabolic
  induction for classical groups}, Amer. J. Math., 142 (2020), pp.~505--546.

\bibitem{Matic3}
{\sc I.~Mati{\'c}}, {\em Strongly positive representations of metaplectic
  groups}, J. Algebra, 334 (2011), pp.~255--274.

\bibitem{Moe2}
{\sc C.~M{\oe}glin}, {\em Paquets stables des s\'eries discr\`etes accessibles
  par endoscopie tordue; leur param\`etre de {L}anglands}, in Automorphic forms
  and related geometry: assessing the legacy of {I}. {I}.
  {P}iatetski-{S}hapiro, vol.~614 of Contemp. Math., Amer. Math. Soc.,
  Providence, RI, 2014, pp.~295--336.

\bibitem{MT1}
{\sc C.~M{\oe}glin and M.~Tadi{\'c}}, {\em Construction of discrete series for
  classical {$p$}-adic groups}, J. Amer. Math. Soc., 15 (2002), pp.~715--786.

\bibitem{Tad9}
{\sc M.~Tadi{\'c}}, {\em Classification of unitary representations in
  irreducible representations of general linear group (non-{A}rchimedean
  case)}, Ann. Sci. \'Ecole Norm. Sup. (4), 19 (1986), pp.~335--382.

\bibitem{Tad5}
\leavevmode\vrule height 2pt depth -1.6pt width 23pt, {\em Structure arising
  from induction and {J}acquet modules of representations of classical
  {$p$}-adic groups}, J. Algebra, 177 (1995), pp.~1--33.

\end{thebibliography}
\end{document}